\documentclass[12pt]{amsart}
\usepackage{amsmath, amsthm, amssymb}
\usepackage[margin=2.5cm]{geometry}
\usepackage{graphicx}
\usepackage[mathscr]{eucal}
\usepackage{hyperref}
\usepackage{todonotes}
\usepackage{stmaryrd}
\usepackage{enumerate}
\usepackage[ruled,vlined]{algorithm2e}

\theoremstyle{plain}
\newtheorem{theorem}{Theorem}
\newtheorem{definition}{Definition}
\newtheorem{proposition}[theorem]{Proposition}
\newtheorem{lemma}[theorem]{Lemma}

\newtheorem{corollary}[theorem]{Corollary}
\theoremstyle{example}
\newtheorem{example}[theorem]{Example}
\theoremstyle{remark}
\newtheorem{remark}[theorem]{Remark}
\newtheorem{Algorithm}[theorem]{Algorithm}

\author{Carmelo Cisto}
\address[Cisto]{Universit\`{a} di Messina, Dipartimento di Scienze  Matematiche e Informatiche, Scienze Fisiche e Scienze della Terra, Viale Ferdinando Stagno D'Alcontres 31, 98166 Messina, Italy}
\email{carmelo.cisto@unime.it}

\title[Cofinite submonoids and ideals of an affine semigroup]{On some properties for cofiniteness of submonoids and ideals of an affine semigroup}

		

\begin{document}

		\begin{abstract}
  	Let $S$ and $\mathcal{C}$ be affine semigroups in $\mathbb{N}^d$ such that $S\subseteq \mathcal{C}$. We provide a characterization for the set $\mathcal{C}\setminus S$ to be finite, together with a procedure and computational tools to check whether such a set is finite and, if so, compute its elements. As a consequence of this result, we provide a characterization for an ideal $I$ of an affine semigroup $S$ so that $S\setminus I$ is a finite set. If so, we provide some procedures to compute the set $S\setminus I$.

		\end{abstract}
\maketitle

\section*{Introduction}		

Let $\mathbb{N}$ be the set of non-negative integers. A monoid $S$ is called an \emph{affine semigroup} if it is a finitely generated submonoid of the additive monoid $\mathbb{N}^d$, for some positive integer $d$. Equivalently, up to isomorphism, $S$ is a monoid that is finitely generated, cancellative, torsion free and reduced. We recall that a monoid $S\subseteq \mathbb{Z}^d$ is finitely generated if there exists a finite set $A\subseteq \mathbb{Z}^d$ contained in $S$ such that $S=\langle A \rangle$, where:

$$\langle A \rangle= \left \lbrace \sum_{i=1}^n\lambda_i \mathbf{a}_i \mid n\in \mathbb{N}, \mathbf{a}_i\in A, \lambda_i\in \mathbb{N}\text{ for all }i\in \{1,\ldots,n\}\right \rbrace .$$ 

\noindent In particular, $S=\langle A \rangle$ is called the monoid generated by $A$ and $A$ is a set of generators of $S$. It is known that for every submonoid $S$ of $\mathbb{N}^d$, there exists a unique minimal set that generates $S$, that is $(S\setminus \{\mathbf{0}\})\setminus (S\setminus \{\mathbf{0}\}+ S\setminus \{\mathbf{0}\})$. In particular, if $S$ is an affine semigroup such a set is finite. 

Let $\mathbb{Q}_+$ be the the set of non-negative rational numbers. If $S=\langle A \rangle\subseteq \mathbb{N}^d$ is an affine semigroup, we consider the sets $\operatorname{cone}(S)=\{ \sum_{i=1}^n q_i \mathbf{a}_i \mid n\in \mathbb{N}, \mathbf{a}_i\in A, q_i\in \mathbb{Q}_+\text{ for all }i\in \{1,\ldots,n\}\}$, that is the \emph{cone} spanned by $S$, and $\operatorname{Group}(S)=\{\mathbf{a}-\mathbf{b} \in \mathbb{Z}^d\mid \mathbf{a},\mathbf{b}\in S\}$, that is called the \emph{group} generated by $S$. The sets $\mathcal{N}(S)=\operatorname{cone}(S)\cap \operatorname{Group}(S)$, called the \emph{normalization} of $S$, and $\mathcal{C}_S=\operatorname{cone}(S)\cap \mathbb{N}^d$ are affine semigroups (by Gordan's Lemma, see for instance \cite[Lemma 2.9]{bruns2009polytopes}). Since $S$ is finitely generated, if $\{\mathbf{g}_1,\ldots,\mathbf{g}_n\}\subseteq \mathbb{N}^d$ is its minimal set of generators, by relabeling if necessary, without loss of generality, we may assume that there exists $r\leq n$ such that $\operatorname{cone}(S)=\operatorname{cone}(\{\mathbf{g}_1,\ldots,\mathbf{g}_r\})$. For minimality, we may also assume that $\operatorname{cone}(\{\mathbf{g}_1,\ldots,\mathbf{g}_r\}\setminus \{\mathbf{g}_i\})\neq \operatorname{cone}(S)$ for each $i\in \{1,\ldots,r\}$. In such a case $\mathbf{g}_i\notin \operatorname{cone}(\{\mathbf{g}_1,\ldots,\mathbf{g}_r\}\setminus \{\mathbf{g}_i\})$ for each  $i\in \{1,\ldots,r\}$ (otherwise $\operatorname{cone}(\{\mathbf{g}_1,\ldots,\mathbf{g}_r\}\setminus \{\mathbf{g}_i\})=\operatorname{cone}(\{\mathbf{g}_1,\ldots,\mathbf{g}_r\}=\operatorname{cone}(S)$). We say that $\{\mathbf{g}_1,\ldots,\mathbf{g}_r\}$ is a set of \emph{extreme rays} of $S$.

 One of the motivations for studying affine semigroups is to extend the notion of numerical semigroup and some its properties. A numerical semigroup is a submnoid $S$ of $\mathbb{N}$ such that $\mathbb{N}\setminus S$ is finite. It is known that every numerical semigroup is a finitely generated monoid (so it is an affine semigroup) and every submonoid of $\mathbb{N}$ is isomorphic to a numerical semigroup. Moreover, it is known that if $S=\langle g_1,\ldots, g_n\rangle \subseteq \mathbb{N}$, then $S$ is a numerical semigroup if and only if $\gcd(g_1,\ldots,g_n)=1$. The monographs \cite{Num-sem-book, rosales2009numerical} are very good references for this and other interesting properties of numerical semigroups. 
 
 A possible way to generalize the notion of numerical semigroup to submonoids in $\mathbb{N}^d$ is to focus on the property of \emph{cofiniteness}. In general, we say that a set $A$ is \emph{cofinite} in $B$ if $A\subseteq B$ and $B\setminus A$ is a finite set. In particular, a numerical semigroup is a cofinite monoid in $\mathbb{N}$. A straightforward step in this direction is to consider submonoids $S$ of $\mathbb{N}^d$ such that $\mathbb{N}^d\setminus S$ is finite. This kind of monoids have been introduced in \cite{failla2016algorithms}, where they are called \emph{generalized numerical semigroups}. The study of the properties of these monoids is still an active area of research (for some recent works, see for instance \cite{W2, Li}). A more general situation is considered in \cite{garcia2018extension}, where the authors introduce the class of monoids $S$ such that $\mathcal{C}_S\setminus S$ is finite. A monoid $S$ of such a family is called a $\mathcal{C}_S$-\emph{semigroup}. In particular, a generalized numerical semigroup is a $\mathcal{C}_S$-semigroup such that $\mathcal{C}_S=\mathbb{N}^d$. Some recent results on $\mathcal{C}$-semigroups are contained, for instance, in \cite{bhardway2023affine, garcia2023some}. It is known that if $S$ is a $\mathcal{C}_S$-semigroup, then $S$ is finitely generated (so, the same occurs when $S$ is a generalized numerical semigroup), that is, the mentioned families of submonoids of $\mathbb{N}^d$ are classes of affine semigroups. 

 A natural problem is to consider a set $A\subseteq \mathbb{N}^d$ and characterize when the set $S=\langle A \rangle$ is a $\mathcal{C}_S$-semigroup. A result of this type has been provided firstly in \cite[Theorem 2.8]{Analele} for generalized numerical semigroups and later in \cite[Theorem 9]{diaz2022characterizing} for the general case of $\mathcal{C}_S$-semigroups. Having in mind these results, in this paper we consider a further generalization for the property of cofiniteness of an affine semigroup $S\subseteq \mathbb{N}^d$, considering the case $S$ is cofinite in another affine semigroup $\mathcal{C}\subseteq \mathbb{N}^d$. In particular, in the main result of this work, we provide a characterization for an affine semigroup $S\subseteq \mathbb{N}^d$ to be cofinite in an affine semigroup $\mathcal{C}\subseteq \mathbb{N}^d$, in terms of the generators of $\mathcal{C}$. We call this kind of monoids $\mathcal{C}$-cofinite. We provide such a result in Section~\ref{sec1}, together with a procedure and computational tools to check whether an affine semigroup is $\mathcal{C}$-cofinite for a fixed affine monoid $\mathcal{C}$ and, if so, to compute the set $\mathcal{C}\setminus S$. 

 As a consequence of the previous result, we obtain a characterization for an ideal of an affine semigroup to be cofinite in it. Recall that, if $S$ is an affine semigroup, an ideal $I$ of $S$ is a subset $I\subseteq S$ such that $I+S\subseteq I$ (in general, for $X,Y\subseteq \mathbb{N}^d$, set $X+Y=\{\mathbf{x}+\mathbf{y}\mid \mathbf{x}\in X, \mathbf{y}\in Y\}$). In Section~\ref{sec2} we give such a characterization so that $S\setminus I$ is finite. Furthermore, for a given affine semigroup $S$ and an ideal $I$ of it, we provide two different approaches in order to check whether $S\setminus I$ is finite and, if so, to compute its elements: the first one uses similar procedures to those suggested for $\mathcal{C}$-cofinite submonoids, the second one uses tools from commutative algebra. Finally, we conclude with some remarks about a relation of this subject with the \emph{Ap\'ery set} of a subset of an affine semigroup.

\section{Cofinite submonoids of an affine semigroup}\label{sec1}

The first aim of this section is to provide a characterization, for a submonoid $S$ of an affine semigroup $\mathcal{C}$, to have finite complement in it. After giving a definition for this kind of monoids and showing that they are also affine semigroups, we introduce some useful tools to prove the mentioned characterization.

\begin{definition}\rm
Let $\mathcal{C}\subseteq \mathbb{N}^d$ be an affine semigroup. A submonoid $S$ of $\mathcal{C}$ is called $\mathcal{C}$-cofinite if the set $\mathcal{C}\setminus S$ is finite.

\end{definition}


\noindent For the next result, we recall that a \emph{term order} $\preceq$ (or \emph{monomial order}) on $\mathbb{N}^d$ is a total order such that $\mathbf{0}\preceq \mathbf{v}$ for all $\mathbf{v}\in \mathbb{N}^d$ and if $\mathbf{v}\preceq \mathbf{w}$, then the inequality $\mathbf{v}+\mathbf{u}\preceq \mathbf{w}+\mathbf{u}$ holds for all $\mathbf{u}\in \mathbb{N}^d$ (see Chapter 2, \S 2 of \cite{CLO} also for some concrete examples of term orders). We use the notation $\mathbf{v}\prec \mathbf{w}$ for $\mathbf{v}\preceq \mathbf{w}$ and $\mathbf{v}\neq \mathbf{w}$.

\begin{proposition}
Let $\mathcal{C}\subseteq \mathbb{N}^d$ be an affine semigroup and $S\subseteq \mathcal{C}$ be a $\mathcal{C}$-cofinite submnoid. Then $S$ is finitely generated.
\end{proposition}

\begin{proof}
Let $\mathcal{C}\setminus S=\{\mathbf{h}_{1}\prec\cdots\prec \mathbf{h}_r\}$ for some term order $\preceq$, with $\mathcal{C}=\langle \mathbf{g}_{1},\mathbf{g}_{2},\ldots,\mathbf{g}_{n}\rangle $. We show that $\mathbf{h}_{1}$ is a minimal generator of $\mathcal{C}$. In fact, suppose $\mathbf{h}_1=\mathbf{a}+\mathbf{b}$ with $\mathbf{a},\mathbf{b}\in \mathcal{C}\setminus \{\mathbf{0}\}$. In particular, $\mathbf{h}_1\neq \mathbf{a}$ and $\mathbf{h}_1\neq \mathbf{b}$. If $\mathbf{h}_1 \prec \mathbf{a}$, then $\mathbf{h}_1\prec \mathbf{h}_1+\mathbf{b}\prec \mathbf{a}+\mathbf{b}$, which contradicts $\mathbf{h}_1=\mathbf{a}+\mathbf{b}$. Hence, $\mathbf{a}\prec \mathbf{h}_1$ and by a similar argument, $\mathbf{b}\prec \mathbf{h}_1$. As a consequence, by the minimality of $\mathbf{h}_1$ we have $\mathbf{a},\mathbf{b}\in S$, obtaining $\mathbf{h}_1\in S$, a contradiction. So, $\mathbf{h}_1$ is a minimal generator of $\mathcal{C}$ and we can suppose $\mathbf{h}_{1}=\mathbf{g}_{1}$. We can argue that $\mathcal{C}\setminus \{\mathbf{h}_1\}=\langle \mathbf{g}_{2},\ldots,\mathbf{g}_{n}, \mathbf{g}_{2}+\mathbf{h}_1,\ldots,\mathbf{g}_{n}+\mathbf{h}_1,2\mathbf{h}_1,3\mathbf{h}_1\rangle$. We can use the same argument to show that $\mathbf{h}_2$ is a minimal generator of $\mathcal{C}\setminus \{\mathbf{h}_1\}$  and so on, until we provide a finite set of generators of $\mathcal{C}\setminus \{\mathbf{h}_1,\ldots,\mathbf{h}_r\}$.
\end{proof}

\noindent Let $\mathcal{C}=\langle \mathbf{g}_{1},\mathbf{g}_{2},\ldots,\mathbf{g}_{n}\rangle \subseteq \mathbb{N}^{d}$ be an affine semigroup. Denote by $\mathbf{e}_1,\ldots,\mathbf{e}_n$ the standard basis vectors of the vector space $\mathbb{R}^n$; in particular, $\mathbb{N}^n=\langle \mathbf{e}_1,\ldots,\mathbf{e}_n \rangle$ as a monoid. We consider the following map:
$$f_\mathcal{C}:\mathbb{N}^{n}\longrightarrow \mathcal{C},\qquad \sum_{i=1}^{n}a_{i}\mathbf{e}_{i}\longmapsto \sum_{i=1}^{n}a_{i}\mathbf{g}_{i}.$$
\noindent Then, $f_\mathcal{C}$ is a monoid epimorphism. Observe that if $\mathbf{a}\in \mathcal{C}$, then $f^{-1}_\mathcal{C}(\mathbf{a})$ is the set of factorizations of $\mathbf{a}$ as combination of vectors in the set $\{\mathbf{g}_{1},\mathbf{g}_{2},\ldots,\mathbf{g}_{n}\}$.

%
%

\begin{lemma}
\label{proje}
Let $\mathcal{C}=\langle \mathbf{g}_{1},\mathbf{g}_{2},\ldots,\mathbf{g}_{n}\rangle \subseteq \mathbb{N}^d$ be an affine semigroup. Suppose $S$ is a submonoid of $\mathcal{C}$. Then 
\begin{enumerate}
\item $\mathcal{C}\setminus S$ is a finite set if and only if $\mathbb{N}^n\setminus f_\mathcal{C}^{-1}(S)$ is a finite set.
\item $\mathcal{C}\setminus S=\{f_\mathcal{C}(\mathbf{h})\mid \mathbf{h}\in \mathbb{N}^n\setminus f_\mathcal{C}^{-1}(S)\}$.
\end{enumerate}
\end{lemma}
\begin{proof}
Denote $T=f_\mathcal{C}^{-1}(S)$. Observe that $T$ is a submonoid of $\mathbb{N}^n$. Moreover, since $f_\mathcal{C}$ is a surjective function, we have $\mathbb{N}^n \setminus T=f_\mathcal{C}^{-1}(\mathcal{C})\setminus f_\mathcal{C}^{-1}(S)=f_\mathcal{C}^{-1}(\mathcal{C}\setminus S)$. 
From this equality, we easily obtain that $|\mathcal{C}\setminus S|\leq |\mathbb{N}^{n}\setminus T|$. Therefore, if $T$ is $\mathbb{N}^n$-cofinite, then $\mathcal{C}\setminus S$ is a finite set. Moreover, for all $\mathbf{x}\in \mathcal{C}$, the set $f_\mathcal{C}^{-1}(\mathbf{x})=\{\mathbf{h}\in \mathbb{N}^n \mid f_\mathcal{C}(\mathbf{h})=\mathbf{x}\}$ is finite. In fact, each element of an affine semigroup has a finite number of factorizations. In particular, if $\mathcal{C}\setminus S$ is finite, then $\mathbb{N}^n \setminus T=\bigcup_{\mathbf{x}\in \mathcal{C}\setminus S}f_\mathcal{C}^{-1}(\mathbf{x})$ is finite. So, claim (1) is proved. Furthermore, by $\mathbb{N}^n \setminus T=f_\mathcal{C}^{-1}(\mathcal{C}\setminus S)$, since $f_\mathcal{C}$ is surjective, we have $f_\mathcal{C}(\mathbb{N}\setminus T)=f_\mathcal{C}(f_\mathcal{C}^{-1}(\mathcal{C}\setminus S))=\mathcal{C}\setminus S$. In particular, we obtain claim (2). 
\end{proof}

We remind that a characterization  for a submonoid $T\subseteq \mathbb{N}^n$ to be $\mathbb{N}^n$-cofinite is provided in \cite[Theorem 2.8]{Analele}. For completeness, we state the result here, since it will be useful later.


\begin{theorem}[\cite{Analele}]
Let $A \subseteq \mathbb{N}^n$. Then $\langle A \rangle$ is $\mathbb{N}^n$-cofinite if and only if the set $A$ fulfills each one of the
following conditions: 
\begin{enumerate}
\item For all $i\in\{1,\ldots,n\}$ there exist $a_1^{(i)}\mathbf{e}_i, a_2^{(i)}\mathbf{e}_i,\ldots, a_{s_i}^{(i)}\mathbf{e}_i\in A$, with $s_i\in \mathbb{N}\setminus \{0\}$, $a_j^{(i)}\in \mathbb{N}\setminus\{0\}$, such that $\gcd (a_1^{(i)},a_2^{(i)},\ldots,a_{s_i}^{(i)})=1$.
\item For each $i,j\in \{1,\ldots,n\}$, with $i\neq j$, there exists $n_{i}^{(j)}\in \mathbb{N}$ such that $\mathbf{e}_{i}+n_{i}^{(j)}\mathbf{e}_{j}\in A$. 
\end{enumerate} 
\label{GNS}
\end{theorem}
\noindent Next, we provide the mentioned characterization for a submonoid $S\subseteq \mathcal{C}$ to be $\mathcal{C}$-cofinite.

\begin{theorem}
Let $\mathcal{C}=\langle \mathbf{g}_{1},\mathbf{g}_{2},\ldots,\mathbf{g}_{n}\rangle \subseteq \mathbb{N}^{d}$ be an affine semigroup and $S\subseteq \mathcal{C}$ be a submonoid. Then $S$ is $\mathcal{C}$-cofinite if and only if the following two conditions are verified:
\begin{enumerate}
\item For all $i\in\{1,\ldots,n\}$, the set $S_{i}=\{\lambda \in \mathbb{N}\mid \lambda\mathbf{g}_{i}\in S\}$ is a numerical semigroup.
\item For each $i,j\in \{1,\ldots,n\}$, with $i\neq j$, there exists $n_{i}^{(j)}\in \mathbb{N}$ such that $\mathbf{g}_{i}+n_{i}^{(j)}\mathbf{g}_{j}\in S$. 
 
\end{enumerate} 
\label{Csem}
\end{theorem}

\begin{proof}
\emph{Necessity}. It is not difficult to see that the set $S_i$ is a monoid for all $i\in \{1,\ldots,n\}$. If the first condition is not verified, then there exists $i$ such that the set $\{\lambda \in \mathbb{N}\mid \lambda\mathbf{g}_{i}\notin S\}\subseteq \mathcal{C}\setminus S$ is not finite. If the second condition is not verified, then for some $i,j$, with $i\neq j$, we have that $\mathbf{g}_{i}+\mu \mathbf{g}_{j}\notin S$ for all $\mu \in \mathbb{N}$. 

\emph{Sufficiency}. We assume that $S_{i}=\langle \lambda_{1}^{(i)}, \lambda_{2}^{(i)},\ldots,\lambda_{s_{i}}^{(i)}\rangle$ is a numerical semigroup for all $i\in \{1,\ldots,n\}$. In particular, $\langle\lambda_{1}^{(i)}\mathbf{g}_{i}, \lambda_{2}^{(i)}\mathbf{g}_i,\ldots,\lambda_{s_{i}}^{(i)}\mathbf{g}_i\rangle \subseteq S$ and $\gcd(\lambda_{1}^{(i)}, \lambda_{2}^{(i)},\ldots,\lambda_{s_{i}}^{(i)})=1$, for each $i\in \{1,\ldots,n\}$. Moreover, for all $i,j\in \{1,\ldots,n\}$, with $i\neq j$, we can consider $n_{i}^{(j)}=\min \{n\in\mathbb{N}\mid \mathbf{g}_{i}+n\mathbf{g}_{j}\in S\}$. For each $i\in \{1,\ldots,n\}$, we set $G_{i}=\{ \lambda_{1}^{(i)}\mathbf{e}_{i},  \lambda_{2}^{(i)}\mathbf{e}_{i}, \ldots, \lambda_{s_{i}}^{(i)}\mathbf{e}_{i}\}\subseteq \mathbb{N}^{n}$. Let $\mathcal{G}=\left(\cup_{i=1}^{n}G_{i}\right) \cup \{\mathbf{e}_{i}+n_{i}^{(j)}\mathbf{e}_{j}\mid i\neq j\}\subseteq \mathbb{N}^n$. The set $\mathcal{G}$ satisfies the conditions of Theorem~\ref{GNS}, so $\langle \mathcal{G}\rangle$ is $\mathbb{N}^n$-cofinite. Consider the monoid $T=f_\mathcal{C}^{-1}(S)\subseteq \mathbb{N}^n$. Observe that, if $\mathbf{x}\in \mathcal{G}$, then $f_\mathcal{C}(\mathbf{x})\in S$, that is, $\mathcal{G}\subseteq T$. In particular, $\langle \mathcal{G} \rangle \subseteq T$ and, as a consequence, $\mathbb{N}^n\setminus T \subseteq \mathbb{N}^n\setminus \langle \mathcal{G}\rangle$. 
Therefore, $T$ is $\mathbb{N}^{n}$-cofinite and, by Lemma~\ref{proje}, $\mathcal{C}\setminus S$ is finite. 
\end{proof}


Let $S$ and $\mathcal{C}$ be affine semigroups, $S\subseteq \mathcal{C}$. We recall that for $\mathcal{C}=\mathcal{C}_S$ a characterization so that $|\mathcal{C}_S\setminus S|<\infty$ has been provided in \cite[Theorem 9]{diaz2022characterizing} and in the same paper the authors provide a procedure to compute $\mathcal{C}_S\setminus S$.

 Theorem~\ref{Csem}, in this form, can be viewed as a generalization of Theorem~\ref{GNS}. The main difference is the following: in the case $\mathcal{C}=\mathbb{N}^d$ (and the same can be trivially considered in the case $\mathcal{C}\cong \mathbb{N}^r$ for some positive integer $r$) the elements satisfying the two conditions of Theorem~\ref{GNS} belong to every set of generators of $S$. In particular, they belong to the minimal set of generators of $S$. In the case $\mathcal{C}\subsetneq \mathbb{N}^d$ (and $\mathcal{C}\ncong \mathbb{N}^r$ for all positive integer $r$) instead we have to look for these elements in the whole semigroup $S$. It is not possible, in general, to consider only a set of the generators of $S$: see, for instance, Example~\ref{example2} and Example~\ref{exa:C2} below, where the elements $3\mathbf{g}_1,\mathbf{g}_1+2\mathbf{g}_3$ belong to $S$ but they do not belong to its set of minimal generators. 

\medskip
\noindent Now, we want to suggest a possible way to verify computationally the two conditions of Theorem~\ref{Csem}. We assume $\mathcal{C}=\langle \mathbf{g}_1,\ldots,\mathbf{g}_n\rangle$ and suppose $S=\langle \mathbf{n}_{1},\mathbf{n}_{2},\ldots, \mathbf{n}_{t}\rangle$. So, we focus on the following:

\begin{enumerate}
\item For all $i\in \{1,\ldots,n\}$, the set $S_{i}=\{\lambda \in \mathbb{N}\mid \lambda\mathbf{g}_{i}\in S\}$ is a numerical semigroup.
\item For each $i,j\in \{1,\ldots,n\}$, $i\neq j$, there exists $n_{i}^{(j)}\in \mathbb{N}$ such that $\mathbf{g}_{i}+n_{i}^{(j)}\mathbf{g}_{j}\in S$. 
\end{enumerate}

\medskip

\noindent \textbf{Condition (1)}. Let $i\in \{1,\ldots,n\}$. Observe that:
$$\{\lambda \in \mathbb{N}\mid \lambda\textbf{g}_{i}\in S\}=\left \lbrace \lambda\in \mathbb{N}\mid \sum_{j=1}^t\lambda_j\mathbf{n}_j+\lambda (-\textbf{g}_{i})=0 ,\text{ for some }\lambda_1,\ldots,\lambda_n \in \mathbb{N}\right\rbrace.$$

\noindent Let $A^{(i)}=[\mathbf{n}_{1},\mathbf{n}_{2},\ldots, \mathbf{n}_{t}, -\mathbf{g}_i]$ be the matrix whose columns are the elements of the finite set of generators of $S$ and the column vector related to $-\mathbf{g}_i\in \mathbb{Z}^d$. So $A^{(i)}$ has $t+1$ columns and entries in $\mathbb{Z}$. We identify every element $\textbf{x}\in \mathbb{N}^{t+1}$ with its column vector and consider the diophantine linear system $A^{(i)}\mathbf{x}=0$. It is easy to see that: 
$$\{\lambda \in \mathbb{N}\mid \lambda\textbf{g}_{i}\in S\}=\{x_{t+1}\mid A^{(i)}\textbf{x}=0 \ \ \mbox{with}\ \mathbf{x}=(x_1,\ldots,x_t,x_{t+1})\in \ \mathbb{N}^{t+1}\}.$$  

\noindent Let $M_i\subseteq \mathbb{N}^{t+1}$ be the set of non-negative integer solutions of the homogeneous diophantine linear system $A^{(i)}\mathbf{x}=0$. It is known that $M_i$ is an affine semigroup in $\mathbb{N}^{t+1}$ (see for instance \cite[Section 1]{rosales1998non-negative}), so there exists a finite set $B_i\subseteq \mathbb{N}^{t+1}$, such that $M_i=\langle B_i \rangle$. Hence, $S_i=\langle \{x_{t+1}\in \mathbb{N}\mid (x_1,\ldots,x_t,x_{t+1})\in B_i\}\rangle $. It is possible to perform these computations in the computer algebra software \texttt{GAP} (\cite{gap}) with the package \texttt{numericalsgps} (\cite{numericalsgps}). It is recommended to use also the package \texttt{NormalizInterface} (that is an interface in \texttt{GAP} for the software \texttt{Normaliz}, see \cite{Normaliz, NormalizInt}), in order to speed up the computation time. 

\begin{example}\label{example2} \rm
Consider the affine semigroups $\mathcal{C}=\langle (1,1),(1,2),(2,1),(3,1)\rangle$ and  $S=\mathcal{C}\setminus \{(1,1),(3,2),(2,3)\}=\langle (1,2),(2,1),(2,2),(3,1),(3,5)\rangle$. Let us compute a finite set of generators of $S_1=\{\lambda \in \mathbb{N}\mid \lambda\textbf{g}_{1}\in S\}$, where $\mathbf{g}_1=(1,1)$. We need to find the non-negative integer solutions of the linear diophantine system $A\mathbf{x}=0$, where $A$ has column vectors $(1,2),(2,1),(2,2),(3,1),(3,5),(-1,-1)$. 

\begin{verbatim}
gap> LoadPackage("num");;
gap> NumSgpsUseNormaliz();;
gap> A:=[[1,2],[2,1],[2,2],[3,1],[3,5],[-1,-1]];;
gap> n:=Length(A);;
gap> A:=TransposedMat(A);
[ [ 1, 2, 2, 3, 3, -1 ], [ 2, 1, 2, 1, 5, -1 ] ]
gap> sol:=HilbertBasisOfSystemOfHomogeneousEquations(A,[]);
[ [ 0, 0, 0, 1, 1, 6 ], [ 0, 0, 1, 0, 0, 2 ], [ 0, 2, 0, 0, 1, 7 ], 
  [ 1, 1, 0, 0, 0, 3 ], [ 2, 0, 0, 1, 0, 5 ] ]
gap> B:=List(sol,i->i[n]);
[ 6, 2, 7, 3, 5 ]
gap> Gcd(B);
1
gap> MinimalGenerators(NumericalSemigroup(B));
[ 2, 3 ]
\end{verbatim} 
Therefore, the previous computations allow to check that $S_1=\langle 2,3\rangle$.
\end{example}
 

In order to verify the first condition of Theorem~\ref{Csem} for elements related to an extreme ray $\mathbf{g}$ of $\mathcal{C}$ having greatest common divisor of its coordinates equals to $1$, it suffices an easier check on a finite set of generators of $S$, as stated by the following result.

\begin{proposition}
Let $\mathcal{C}=\langle \mathbf{g}_{1},\mathbf{g}_{2},\ldots,\mathbf{g}_{r},\mathbf{g}_{r+1},\mathbf{g}_{r+2},\ldots,\mathbf{g}_{r+m}\rangle\subseteq \mathbb{N}^d$ be an affine semigroup such that $\operatorname{cone}(\mathcal{C})=\operatorname{cone}(\{\mathbf{g}_1,\ldots,\mathbf{g}_r\})$. Assume that $\mathbf{g}_i\notin \operatorname{cone}(\{\mathbf{g}_1,\ldots,\mathbf{g}_r\}\setminus \{\mathbf{g}_i\})$ for some $i\in \{1,\ldots,r\}$ and that the greatest common divisor of the coordinates of $\mathbf{g}_i$ is $1$. Suppose $S\subseteq \mathcal{C}$ is a submonoid and denote by $A$ the minimal set of generators of $S$. Consider the monoid $S_{i}=\{\lambda \in \mathbb{N}\mid \lambda\mathbf{g}_{i}\in S\}\subseteq \mathbb{N}$ and let $\{\lambda_{1}^{(i)}, \lambda_{2}^{(i)},\ldots,\lambda_{s_{i}}^{(i)}\}$ be the minimal set of generators of $S_i$.
Then $\{\lambda_{1}^{(i)}\mathbf{g}_{i}, \lambda_{2}^{(i)}\mathbf{g}_{i},\ldots,\lambda_{s_{i}}^{(i)}\mathbf{g}_{i}\}\subseteq A$. 
\end{proposition}

\begin{proof}
Let $\ell\in \{1,\ldots,s_i\}$. We want to prove that $\lambda_{\ell}^{(i)} \mathbf{g}_{i}\in A$. First of all, observe that if $\lambda_{\ell}^{(i)}\mathbf{g}_{i} =\sum_{j=1}^{r}\beta_{j}\mathbf{g}_{j}$ with $\beta_{j} \in \mathbb{Q}_{\geq 0}$, then $\beta_j=0$ for all $j\neq i$. In fact, if there exists $k\in \{1,\ldots,r\}\setminus \{i\}$ such that $\beta_k\neq 0$, then $\beta_i<\lambda_{\ell}^{(i)}$. In particular, $(\lambda_{\ell}^{(i)} -\beta_i)\mathbf{g}_i=\sum_{j=1,j\neq i}^{r}\beta_{j}\mathbf{g}_{j}$, from which we easily obtain that $\mathbf{g}_i\in \operatorname{cone}(\{\mathbf{g}_1,\ldots,\mathbf{g}_r\}\setminus \{\mathbf{g}_i\})$, a contradiction. So, the only possibility is $\beta_i=\lambda$ and $\beta_j=0$ for all $j\in \{1,\ldots,r\}\setminus \{i\}$.\\
\noindent We know that $\lambda_{\ell}^{(i)} \mathbf{g}_{i}\in \langle A \rangle=S$ and suppose that $\lambda_{\ell}^{(i)} \mathbf{g}_{i}\notin A$. Hence, $\lambda_{\ell}^{(i)} \mathbf{g}_{i}=\sum_{\mathbf{a}\in A}\mu_{\mathbf{a}}\mathbf{a}$ for some $\mu_{\mathbf{a}}\in \mathbb{N}$ such that $\sum_{\mathbf{a}\in A}\mu_\mathbf{a}> 1$. For every $\mathbf{a}\in A$, we have $\mathbf{a}\in \mathcal{C}$. In particular, $\mathbf{a}\in \operatorname{cone}(\{\mathbf{g}_1,\ldots,\mathbf{g}_r\})$. For $\mathbf{a}\in A$, assume $\mathbf{a}=\sum_{j=1}^r \gamma_{(\mathbf{a},j)} \mathbf{g}_j$ with $\gamma_{(\mathbf{a},j)}\in \mathbb{Q}_{\geq 0}$ for each $j\in \{1,\ldots,r\}$. As a consequence, $\lambda_{\ell}^{(i)} \mathbf{g}_{i}=\sum_{\mathbf{a}\in A}\mu_{\mathbf{a}}\mathbf{a}=\sum_{j=1}^r(\sum_{\mathbf{a}\in A}\mu_\mathbf{a} \gamma_{(\mathbf{a},j)}) \mathbf{g}_j$. By the discussion at the beginning of the proof, we have $\sum_{\mathbf{a}\in A}\mu_\mathbf{a} \gamma_{(\mathbf{a},j)}=0$ for all $j\in \{1,\ldots,r\}\setminus \{i\}$. So, for all $j\in \{1,\ldots,r\}\setminus \{i\}$, since $\mu_\mathbf{a} \gamma_{(\mathbf{a},j)}\geq 0$, the only possibility is $\mu_\mathbf{a} \gamma_{(\mathbf{a},j)}= 0$. In particular, if $\mu_\mathbf{a}\neq 0$ we have $\gamma_{(\mathbf{a},j)}= 0$ for all $j\in \{1,\ldots,r\}\setminus \{i\}$. 
Then, for all $\mathbf{a}\in A$ with $\mu_\mathbf{a}\neq 0$, we have $\mathbf{a}=\gamma_{(\mathbf{a},i)} \mathbf{g}_i$ for some $ \gamma_{(\mathbf{a},i)}\in \mathbb{Q}_{\geq 0}\setminus \{0\}$. Since $\mathbf{g}_i,\mathbf{a}\in \mathbb{N}^d$, the hypothesis that the greatest common divisor of the coordinates of $\mathbf{g}_i$ is $1$ forces $ \gamma_{(\mathbf{a},i)}\in \mathbb{N}\setminus \{0\}$. Therefore, for all $\mathbf{a}\in A$ with $\mu_\mathbf{a}\neq 0$, we have $\gamma_{(\mathbf{a},i)}\in S_i= \langle \lambda_{1}^{(i)}, \lambda_{2}^{(i)} ,\ldots,\lambda_{s_{i}}^{(i)} \rangle$. Let $\mathbf{b}\in A$ such that $\mu_\mathbf{b}\neq 0$, and assume  $\gamma_{(\mathbf{b},i)}=\sum_{j=1}^{s_i}\alpha_j \lambda_{j}^{(i)}$, $\alpha_j\in \mathbb{N}$. Then we can write $\lambda_{\ell}^{(i)} \mathbf{g}_i=\sum_{\mathbf{a}\in A, \mathbf{a}\neq \mathbf{b}}\mu_\mathbf{a}\mathbf{a}+\mu_\mathbf{b}\left(\sum_{j=1,j\neq \ell}^{s_i}\alpha_j \lambda_{j}^{(i)}\mathbf{g}_i\right)+ \mu_\mathbf{b}\alpha_\ell \lambda_{\ell}^{(i)}\mathbf{g}_i$. Since $\mu_\mathbf{b}> 0$ and $\sum_{\mathbf{a}\in A}\mu_\mathbf{a}>1$, the only possibility is $\alpha_\ell=0$. In particular, we can argue that for all $\mathbf{a}\in A$, with $\mu_\mathbf{a}\neq 0$, $\gamma_{(\mathbf{a},i)}\in \langle \{\lambda_{1}^{(i)}, \lambda_{2}^{(i)} ,\ldots,\lambda_{s_{i}}^{(i)}\}\setminus \{\lambda_{\ell}^{(i)}\} \rangle $. As a consequence, $\lambda_{\ell}^{(i)} \mathbf{g}_i\in \langle \{\gamma_{(\mathbf{a},i)} \mathbf{g}_i\mid \mathbf{a}\in A, \mu_\mathbf{a}\neq 0\} \rangle\subseteq \langle \{\lambda_{1}^{(i)}\mathbf{g}_{i}, \lambda_{2}^{(i)} ,\ldots,\lambda_{s_{i}}^{(i)}\mathbf{g}_{i}\}\setminus \{\lambda_{\ell}^{(i)} \mathbf{g}_i\}\rangle$. Hence, $\lambda_{\ell}^{(i)} \in \langle \{\lambda_{1}^{(i)}\, \lambda_{2}^{(i)} ,\ldots,\lambda_{s_{i}}^{(i)}\}\setminus \{\lambda_{\ell}^{(i)}\}\rangle$, but this contradicts the fact that $S_i$ is minimally generated by $\{\lambda_{1}^{(i)}\mathbf{g}_{i}, \lambda_{2}^{(i)}\mathbf{g}_i ,\ldots,\lambda_{s_{i}}^{(i)}\mathbf{g}_{i}\}$. So we can conclude $\lambda_{\ell}^{(i)}\mathbf{g}_i\in A$.
\end{proof}

\medskip 
\noindent \textbf{Condition (2)}. Firstly observe that, in the outlined framework, for the second condition of Theorem~\ref{Csem} we can point out that: if $\mathbf{g}_i\in S$ for some $i\in \{1,\ldots,n\}$, for what concerns the elements of the form $\mathbf{g}_i+n_i^{(j)}\mathbf{g}_j$ we can consider $n_i^{(j)}=0$ for all $j\in \{1,\ldots,n\}\setminus\{i\}$. So, assume $\mathbf{g}_i\notin S$. 
For $j\in \{1,\ldots,n\}\setminus \{i\}$, consider that:
$$\{k\in \mathbb{N}\mid \textbf{g}_{i}+k\textbf{g}_{j}\in S\}= \left \lbrace k\in \mathbb{N}\mid \mathbf{g}_i=\sum_{j=1}^t\lambda_j\mathbf{n}_j+k (-\textbf{g}_{j}) ,\text{ for some }\lambda_1,\ldots,\lambda_n \in \mathbb{N}\right\rbrace.$$  
So, in this case, we need to find all non-negative integer factorizations of $\mathbf{g}_i$ in the monoid $\langle \mathbf{n}_1,\ldots, \mathbf{n}_t,-\mathbf{g}_j \rangle \subseteq \mathbb{Z}^{t+1}$ and take, for each factorization, the coefficient of $-\mathbf{g}_j$.  In particular, this problem is related to find the non-negative integer solutions of a non homogeneus diophantine linear system of equations, that is, using the notation introduced in Condition (1), we have:
$$\{k\in \mathbb{N}\mid \textbf{g}_{i}+k\textbf{g}_{j}\in S\}=\{x_{t+1}\mid A^{(j)}\textbf{x}=\mathbf{g}_i \ \ \mbox{with}\ \mathbf{x}=(x_1,\ldots,x_t,x_{t+1})\in \ \mathbb{N}^{t+1}\}.$$  
Recall that, if $R_{ij}$ is the set of non-negative integer solutions of $A^{(j)}\textbf{x}=\mathbf{g}_i$ and $V_{ij}$ is the set of minimal elements of $R_{ij}$, with respect to the natural partial order in $\mathbb{N}^{t+1}$, then the set $V_{ij}$ is finite and $R_{ij}=\bigcup_{\mathbf{y}\in V_{ij}}(\mathbf{y}+M_j)$,  where $M_j$ is the set of non-negative integer solutions of the homogenenus diophantine linear system $A^{(j)}\mathbf{x}=0$ (see for instance \cite[Theorem 5.2]{pison2004nsolutions}). In particular, $\min \{k\in \mathbb{N}\mid \textbf{g}_{i}+k\textbf{g}_{j}\in S\}=\min\{x_{t+1}\mid (x_1,\ldots,x_t,x_{t+1})\in V_{ij}\}$. The following example shows how to perform such a computation with \texttt{GAP}, using \texttt{numericalsgps} and \texttt{NormalizInterface}.

\begin{example}\label{exa:C2} \rm
Consider the semigroups $\mathcal{C}$ and $S$ as in Example~\ref{example2}. We focus on the element such as $\mathbf{g}_1+n_{1}^{(3)}\mathbf{g}_3$. In order to compute 
$\{k\in \mathbb{N}\mid \textbf{g}_{1}+k\textbf{g}_{3}\in S\}$, with $\mathbf{g}_1=(1,1)$ and $\mathbf{g}_3=(2,1)$, we need to find the minimal factorizations of $(1,1)$ with respect to the set $\{(1,2),(2,1),(2,2),(3,1),(3,5),(-2,-1)\}$. That is, we need to find the set $V_{13}$ introduced before.  
\begin{verbatim}
gap> LoadPackage("num");;
gap> NumSgpsUseNormaliz();;
gap> A:=[[1,2],[2,1],[2,2],[3,1],[3,5],[-2,-1]];;
gap> n:=Length(A);;
gap> F:=FactorizationsVectorWRTList([1,1], A);
[ [ 0, 0, 0, 6, 1, 10 ], [ 0, 0, 1, 1, 0, 2 ], [ 1, 0, 0, 2, 0, 3 ] ]
gap> List(F,i->i[n]);
[ 10, 2, 3 ]
\end{verbatim}
The computations above show that $V_{13}=\{(0,0,0,6,1,10),(0,0,1,1,0,2),(1,0,0,2,0,3)\}$. The package manual of \texttt{numericalagps} explains that, if \texttt{v} is a list of non-negative integers and \texttt{ls} is a list of lists of non-negative integers, then the function \texttt{FactorizationsVectorWRTList( v, ls )} returns the set of factorizations of \texttt{v} in terms of the elements of \texttt{ls}. Actually, when \texttt{NormalizInterface} is used, that function works also in the case \texttt{ls} has vectors with negative coordinates. In fact, by the code of that function, using \texttt{NormalizInterface}, the function computes exactly the minimal elements (with respect to the natural partial order) of the set of  non-negative integer solutions of the system \texttt{ls}*\texttt{x}=\texttt{v}, in the case the system admits solutions (each list of integers is considered here as a column vector).
So, in this case, $\min \{k\in \mathbb{N}\mid \textbf{g}_{1}+k\textbf{g}_{3}\in S\}=2$ and we can consider $n_1^{(3)}=2$, that is, $\mathbf{g}_1+2\mathbf{g}_3\in S$.
\end{example}

Now, we can describe a procedure to check if $\mathcal{C}\setminus S$ is finite and, if so, to compute its elements. By Lemma~\ref{proje}, a direct way is to focus on the monoid $f_\mathcal{C}^{-1}(S)$. In particular, we start by finding a set of generators of $f_\mathcal{C}^{-1}(S)$. \\
\noindent Assume $\mathcal{C}=\langle  \mathbf{g}_{1},\mathbf{g}_{2},\ldots,\mathbf{g}_{n}\rangle$ and $S=\langle \mathbf{n}_{1},\mathbf{n}_{2},\ldots,\mathbf{n}_{t} \rangle $. Observe that an element of the form $\sum_{i=1}^n x_i \mathbf{e}_i$ belongs to $f_\mathcal{C}^{-1}(S)$ if and only if there exist $\lambda_1,\ldots,\lambda_t \in \mathbb{N}$ such that $\sum_{i=1}^n x_i \mathbf{g}_i=\sum_{j=1}^t \lambda_j \mathbf{n}_j$. Consider the matrix $A^{(\mathcal{C},S)}=[\mathbf{g}_{1},\mathbf{g}_{2},\ldots,\mathbf{g}_{n}, - \mathbf{n}_{1},-\mathbf{n}_{2},\ldots,-\mathbf{n}_{t}]$, obtained identifying every integer vector with a column vector. Then, the elements of $f_\mathcal{C}^{-1}(S)$ can be obtained from the non-negative integer solutions of the homogeneous diophantine linear system $A^{(\mathcal{C},S)}\mathbf{x}=\mathbf{0}$, that is:
$$ f_\mathcal{C}^{-1}(S)=\{(x_1,\ldots,x_n)\in \mathbb{N}^d \mid A^{(\mathcal{C},S)}\mathbf{x}=\mathbf{0}\text{ with }\mathbf{x}=(x_1,\ldots,x_n,x_{n+1},\ldots,x_{n+t})\in \mathbb{N}^{n+t}\}.$$
Let $M_{(\mathcal{C},S)}\subseteq \mathbb{N}^{n+t}$ be the set of non-negative integer solutions of the homogeneous diophantine linear system $A^{(\mathcal{C},S)}\mathbf{x}=0$. The set $M_{(\mathcal{C},S)}$ is an affine semigroup in $\mathbb{N}^{n+t}$ (see \cite[Section 1]{rosales1998non-negative}). So, there exists a finite set $B_{(\mathcal{C},S)}\subseteq \mathbb{N}^{n+t}$, such that $M_{(\mathcal{C},S)}=\langle B_{(\mathcal{C},S)} \rangle$. Hence, $f_\mathcal{C}^{-1}(S)=\langle \{(x_1,\ldots, x_{n})\in \mathbb{N}^n\mid (x_1,\ldots,x_n,x_{n+1},\ldots,x_{n+t})\in B_{(\mathcal{C},S)}\}\rangle $. These computations can be performed using the computer algebra software \texttt{GAP}  with the packages \texttt{numericalsgps} and \texttt{NormalizInterface}. \\
\noindent Now, by Lemma~\ref{proje}, we need to check if $\mathbb{N}^n\setminus f_\mathcal{C}^{-1}(S)$ is finite and, if so, compute its elements. By the previous arguments we obtained a finite set $B\subseteq \mathbb{N}^{n}$ such that $f_\mathcal{C}^{-1}(S)=\langle B \rangle$. So, we can test if $\mathbb{N}^n\setminus f_\mathcal{C}^{-1}(S)$ is finite by Theorem~\ref{GNS}. Once we check it is finite, in order to compute the set $\mathbb{N}^n\setminus f_\mathcal{C}^{-1}(S)$, we can consider the procedure described in \cite{cisto2021algorithms} (it can be performed using \texttt{GAP}  with the package \texttt{numericalsgps}). Finally, we obtain the set $\mathcal{C}\setminus S$ considering all elements $f_\mathcal{C}(\mathbf{h})$, for every $\mathbf{h}\in \mathbb{N}^n \setminus f_\mathcal{C}^{-1}(S)$.
\begin{Algorithm}\label{alg:CminusS2}
Let $\mathcal{C}=\langle \mathbf{g}_{1},\mathbf{g}_{2},\ldots,\mathbf{g}_{n}\rangle \subseteq \mathbb{N}^{d}$ be an affine semigroup and $S$ a submonoid of $\mathcal{C}$. Suppose that $S=\langle \mathbf{n}_{1},\mathbf{n}_{2},\ldots, \mathbf{n}_{t}\rangle$. In order to compute $\mathcal{C}\setminus S$ we can consider the following steps:
\begin{enumerate}
\item Consider the matrix $A^{(\mathcal{C},S)}=[\mathbf{g}_{1},\mathbf{g}_{2},\ldots,\mathbf{g}_{n}, - \mathbf{n}_{1},-\mathbf{n}_{2},\ldots,-\mathbf{n}_{t}]$, where each element is identified as a column vector.
\item Compute a finite set $B_{(\mathcal{C},S)}\subseteq \mathbb{N}^{n+t}$, such that $\langle B_{(\mathcal{C},S)} \rangle$ is the set of non-negative integer solutions of the homogeneous diophantine linear system $A^{(\mathcal{C},S)}\mathbf{x}=0$.
\item Set $B=\{(x_1,\ldots, x_{n})\in \mathbb{N}^n\mid (x_1,\ldots,x_n,x_{n+1},\ldots,x_{n+t})\in B_{(\mathcal{C},S)}\}$.
\item Check if the set $B$ satisfies the conditions of Theorem~\ref{GNS}, that is, check if $\langle B \rangle$ is $\mathbb{N}^n$-cofinite.
\item If $\langle B \rangle$ is not $\mathbb{N}^n$-cofinite, then $S$ is not $\mathcal{C}$-cofinite.
\item If $\langle B \rangle$ is $\mathbb{N}^n$-cofinite, compute $H=\mathbb{N}^n\setminus \langle B \rangle$. 
\item Compute $\mathcal{C}\setminus S=\{f_\mathcal{C}(\mathbf{h})\mid \mathbf{h}\in H\}$.
\end{enumerate}
\end{Algorithm}

\begin{example}\label{example5} \rm
Consider the affine semigroups $\mathcal{C}=\langle (1,1),(1,2),(2,1),(3,1)\rangle$ and  $S=\mathcal{C}\setminus \{(1,1),(3,2),(2,3)\}=\langle (1,2),(2,1),(2,2),(3,1),(3,5)\rangle$ (as in the previous examples). In the following, we show how it is possible to perform Algorithm~\ref{alg:CminusS2} using the computer algebra software \texttt{GAP}, with the packages \texttt{numericalsgps} and \texttt{NormalizInterface}.
\begin{verbatim}
gap> LoadPackage("num");;
gap> NumSgpsUseNormaliz();;
gap> C:=[[1,1],[1,2],[2,1],[3,1]];
gap> n:=Length(C);
4
[ [ 1, 1 ], [ 1, 2 ], [ 2, 1 ], [ 3, 1 ] ]
gap> S:=[[1,2],[2,1],[2,2],[3,1],[3,5]];
[ [ 1, 2 ], [ 2, 1 ], [ 2, 2 ], [ 3, 1 ], [ 3, 5 ] ]
gap> A:=Concatenation(C,-S);
[ [ 1, 1 ], [ 1, 2 ], [ 2, 1 ], [ 3, 1 ], [ -1, -2 ], [ -2, -1 ], [ -2, -2 ], 
  [ -3, -1 ], [ -3, -5 ] ]
gap> A:=TransposedMat(A);
[ [ 1, 1, 2, 3, -1, -2, -2, -3, -3 ], [ 1, 2, 1, 1, -2, -1, -2, -1, -5 ] ]
gap> gap> sol:=HilbertBasisOfSystemOfHomogeneousEquations(A,[]);
[ [ 0, 0, 0, 1, 0, 0, 0, 1, 0 ], [ 0, 0, 1, 0, 0, 1, 0, 0, 0 ], 
  [ 0, 0, 4, 0, 0, 0, 1, 2, 0 ], [ 0, 0, 5, 0, 1, 0, 0, 3, 0 ], 
  [ 0, 0, 12, 0, 0, 0, 0, 7, 1 ], [ 0, 1, 0, 0, 1, 0, 0, 0, 0 ], 
  [ 0, 1, 0, 1, 0, 1, 1, 0, 0 ], [ 0, 1, 0, 3, 0, 5, 0, 0, 0 ], 
  [ 0, 1, 3, 0, 0, 0, 2, 1, 0 ], [ 0, 1, 7, 0, 0, 0, 0, 4, 1 ], 
  [ 0, 2, 2, 0, 0, 0, 0, 1, 1 ], [ 0, 2, 2, 0, 0, 0, 3, 0, 0 ], 
  [ 0, 3, 0, 2, 0, 3, 0, 0, 1 ], [ 0, 3, 1, 0, 0, 0, 1, 0, 1 ], 
  [ 0, 3, 1, 1, 0, 0, 4, 0, 0 ], [ 0, 4, 0, 1, 0, 0, 2, 0, 1 ], 
  [ 0, 4, 0, 2, 0, 0, 5, 0, 0 ], [ 0, 5, 0, 1, 0, 1, 0, 0, 2 ], 
  [ 0, 7, 1, 0, 0, 0, 0, 0, 3 ], [ 0, 8, 0, 1, 0, 0, 1, 0, 3 ], 
  [ 0, 12, 0, 1, 0, 0, 0, 0, 5 ], [ 1, 0, 0, 1, 0, 2, 0, 0, 0 ], 
  [ 1, 0, 2, 0, 0, 0, 1, 1, 0 ], [ 1, 0, 3, 0, 1, 0, 0, 2, 0 ], 
  [ 1, 0, 10, 0, 0, 0, 0, 6, 1 ], [ 1, 1, 1, 0, 0, 0, 2, 0, 0 ], 
  [ 1, 1, 5, 0, 0, 0, 0, 3, 1 ], [ 1, 2, 0, 0, 0, 0, 0, 0, 1 ], 
  [ 1, 2, 0, 1, 0, 0, 3, 0, 0 ], [ 2, 0, 0, 0, 0, 0, 1, 0, 0 ], 
  [ 2, 0, 1, 0, 1, 0, 0, 1, 0 ], [ 2, 0, 8, 0, 0, 0, 0, 5, 1 ], 
  [ 2, 1, 3, 0, 0, 0, 0, 2, 1 ], [ 3, 0, 0, 0, 1, 1, 0, 0, 0 ], 
  [ 3, 0, 6, 0, 0, 0, 0, 4, 1 ], [ 3, 1, 1, 0, 0, 0, 0, 1, 1 ], 
  [ 4, 0, 4, 0, 0, 0, 0, 3, 1 ], [ 4, 1, 0, 0, 0, 1, 0, 0, 1 ], 
  [ 5, 0, 0, 0, 2, 0, 0, 1, 0 ], [ 5, 0, 2, 0, 0, 0, 0, 2, 1 ], 
  [ 6, 0, 0, 0, 0, 0, 0, 1, 1 ], [ 7, 0, 0, 0, 0, 2, 0, 0, 1 ] ]
gap> B:=List(sol,i->i{[1..n]});
[ [ 0, 0, 0, 1 ], [ 0, 0, 1, 0 ], [ 0, 0, 4, 0 ], [ 0, 0, 5, 0 ], 
  [ 0, 0, 12, 0 ], [ 0, 1, 0, 0 ], [ 0, 1, 0, 1 ], [ 0, 1, 0, 3 ], 
  [ 0, 1, 3, 0 ], [ 0, 1, 7, 0 ], [ 0, 2, 2, 0 ], [ 0, 2, 2, 0 ], 
  [ 0, 3, 0, 2 ], [ 0, 3, 1, 0 ], [ 0, 3, 1, 1 ], [ 0, 4, 0, 1 ], 
  [ 0, 4, 0, 2 ], [ 0, 5, 0, 1 ], [ 0, 7, 1, 0 ], [ 0, 8, 0, 1 ], 
  [ 0, 12, 0, 1 ], [ 1, 0, 0, 1 ], [ 1, 0, 2, 0 ], [ 1, 0, 3, 0 ], 
  [ 1, 0, 10, 0 ], [ 1, 1, 1, 0 ], [ 1, 1, 5, 0 ], [ 1, 2, 0, 0 ], 
  [ 1, 2, 0, 1 ], [ 2, 0, 0, 0 ], [ 2, 0, 1, 0 ], [ 2, 0, 8, 0 ], 
  [ 2, 1, 3, 0 ], [ 3, 0, 0, 0 ], [ 3, 0, 6, 0 ], [ 3, 1, 1, 0 ], 
  [ 4, 0, 4, 0 ], [ 4, 1, 0, 0 ], [ 5, 0, 0, 0 ], [ 5, 0, 2, 0 ], 
  [ 6, 0, 0, 0 ], [ 7, 0, 0, 0 ] ]
gap> T:=AffineSemigroup(B);
<Affine semigroup in 4 dimensional space, with 41 generators>
gap> MinimalGenerators(T);
[ [ 0, 0, 0, 1 ], [ 0, 0, 1, 0 ], [ 0, 1, 0, 0 ], [ 1, 0, 0, 1 ], 
  [ 1, 0, 2, 0 ], [ 1, 1, 1, 0 ], [ 1, 2, 0, 0 ], [ 2, 0, 0, 0 ], 
  [ 3, 0, 0, 0 ] ]
gap> H:=Gaps(T);
[ [ 1, 0, 0, 0 ], [ 1, 0, 1, 0 ], [ 1, 1, 0, 0 ] ]
gap> Set(List(H,i->i*C));
[ [ 1, 1 ], [ 2, 3 ], [ 3, 2 ] ]
\end{verbatim} 
Therefore, the previous computations show that $B_{(\mathcal{C},S)}$=\{(0,0,0,1),(0,0,1,0),(0,1,0,0),(1,0,0,1), (1,0,2,0),(1,1,1,0),(1,2,0,0),(2,0,0,0),(3,0,0,0)\}, $\mathbb{N}^n \setminus f_\mathcal{C}^{-1}(S)=\{(1,0,0,0),(1,0,1,0),(1,1,0,0)\}$ and $\mathcal{C}\setminus S=\{(1,1),(2,3),(3,2)\}$. 
\end{example}

\section{Cofinite ideals of an affine semigroup}\label{sec2}

Let $S\subseteq \mathbb{N}^{d}$ be an affine semigroup. A set $I$ is an \emph{ideal} of $S$ if $I\subseteq S$ and $I+S\subseteq I$. Every ideal $I$ can be expressed as $I=X+S$, where $X=\operatorname{Minimals}_{\leq_S}(I)$ and $\leq_S$ is the partial order in $\mathbb{N}^d$ defined by $\mathbf{x}\leq_S \mathbf{y}$ if $\mathbf{y}-\mathbf{x}\in S$. In particular, the set $X$ is called a set of generators of $I$. Furthermore the set $X$ is finite (see for instance \cite[Proposition 2.7.4]{g-hk}), that is, every ideal of an affine semigroup is \emph{finitely generated}. We simply denote by $\leq$ the order $\leq_{\mathbb{N}^n}$, that is, the natural partial order in $\mathbb{N}^n$, $n$ any positive integer.

\noindent Theorem~\ref{Csem} can be used to characterize when $S\setminus I$ is finite, obtaining the following result.

\begin{theorem}\label{thm:cofinite-ideal}
Let $S\subseteq \mathbb{N}^d$ be an affine semigroup generated by $\{\mathbf{g}_1,\ldots,\mathbf{g}_n\}$ and $I$ be an ideal of $S$. Then $S\setminus I$ is a finite set if and only if for all $i\in \{1,\ldots,n\}$ there exists $k_i\in \mathbb{N}\setminus\{0\}$ such that $k_i\mathbf{g}_i\in I$.
\end{theorem}
\begin{proof}

Observe that $I\cup \{\mathbf{0}\}$ is a submonoid of $S$, so we can use Theorem~\ref{Csem}.

\emph{Necessity.} If $S\setminus I$ is finite, from condition $(1)$ of Theorem~\ref{Csem} we have that the set $\{\lambda\in \mathbb{N}\mid \lambda \mathbf{g}_i\in I\}$ is a numerical semigroup for all $i\in \{1,\ldots,n\}$, in particular this set contains nonzero integers. So, for all $i\in \{1,\ldots,n\}$ there exists $k_i\in \mathbb{N}\setminus \{0\}$ such that $k_i\mathbf{g}_i \in I$.

\emph{Sufficiency.}  Suppose that for all $i\in \{1,\ldots,n\}$ there exists $k_i\in \mathbb{N}\setminus \{0\}$ such that $k_i\mathbf{g}_i \in I$. It suffices to check that $I\cup \{\mathbf{0}\}$ satisfies both the conditions of Theorem~\ref{Csem}. Since $k_i\mathbf{g}_i \in I$ we obtain  $k_i\mathbf{g}_i +\lambda \mathbf{g}_i\in I$ for all $\lambda\in \mathbb{N}$, in particular $\{\lambda\in \mathbb{N}\mid \lambda \mathbf{g}_i\in I\}\supseteq \mathbb{N}\setminus \{1,\ldots,k_i-1\}$, that is, the first condition is satisfied. The second condition holds trivially by the definition of ideal, since for all $i,j\in \{1,\ldots,n\}$, with $i\neq j$, we have $\mathbf{g}_i+k_j\mathbf{g}_j \in I$.
\end{proof}

\begin{Algorithm}\label{alg:2}
Let $S\subseteq \mathbb{N}^d$ be an affine semigroup generated by $\{\mathbf{g}_1,\ldots,\mathbf{g}_n\}$ and $I$ be an ideal of $S$, such that $S\setminus I$ is a finite set. Then, it is possible to compute the set $S\setminus I$ by the following two steps:
\begin{enumerate}
\item For all $i\in \{1,\ldots,n\}$ compute $k_i=\min \{k\in\mathbb{N}\mid k\mathbf{g}_i\in I\}$. If for some $i\in\{1,\ldots,n\}$ we have $\{k\in\mathbb{N}\mid k\mathbf{g}_i\in I\}=\emptyset$, then $S\setminus I$ is not finite.
\item Let $P=\{(\lambda_1,\ldots,\lambda_n)\in \mathbb{N}^n \mid \lambda_i<k_i\ \mbox{for all}\ i\in \{1,\ldots,n\}\}$ (observe that $P$ is a finite set). Then $$S\setminus I=\left\lbrace\sum_{i=1}^n \lambda_i \mathbf{g}_i \notin I\mid (\lambda_1,\ldots,\lambda_n)\in P\right \rbrace.$$
\end{enumerate}
\end{Algorithm}

\begin{remark}
Observe that for step (1) of Algorithm~\ref{alg:2} it would suffice to find $k_i\in \{k\in\mathbb{N}\mid k\mathbf{g}_i\in I\}$, not necessarily being the minimum. But it is better to find $k_i$ as the minimum in order to have $P$ with smallest possible cardinality.
\end{remark}

\noindent Let $S\subseteq \mathbb{N}^d$ be an affine semigroup generated by $\{\mathbf{g}_1,\ldots,\mathbf{g}_n\}$, $I$ be an ideal of $S$ and assume $I=X+S$ with $X=\{\mathbf{u}_1,\ldots,\mathbf{u}_r\}\subset S$. For $i\in \{1\,\ldots,n\}$, in order to compute the integer $k_i=\min \{k\in\mathbb{N}\mid k\mathbf{g}_i\in I\}$, or to test if it exists, consider that: for $k\in \mathbb{N}$, then $k\mathbf{g}_i\in I$ if and only if there exists $j\in \{1,\ldots,r\}$ such that $k\mathbf{g}_i=x_j+\sum_{l=1}^n\lambda_l \mathbf{g}_l$ with $\lambda_l\in \mathbb{N}$ for each $l\in \{1,\ldots,n\}$. Equivalently, $-\mathbf{u}_j=\sum_{l=1,l\neq i}\lambda_l \mathbf{g}_l+k(-\mathbf{g}_i)$. So, in this case, we need to find all non-negative integer factorizations of $-\mathbf{u}_j$ in the monoid $\langle \{\mathbf{g}_1,\ldots, \mathbf{g}_n\}\setminus \{\mathbf{g}_i\}\cup \{-\mathbf{g}_i\} \rangle \subseteq \mathbb{Z}^{n}$ and take, for each factorization, the coefficient of $-\mathbf{g}_i$. This computation can be performed by \texttt{GAP} as in Example~\ref{exa:C2}. Moreover, the package \texttt{numericalsgps} also contains many routines to deal with ideals of an affine semigroup. For instance, it is possible to test if an integer vector belongs to an ideal or not.

\medskip
To test if $S\setminus I$ is finite and, if so, to compute it, we can also consider the map $f_S: \mathbb{N}^n \longrightarrow S$ introduced in Section~\ref{sec1}, as explained in the next result.

\begin{proposition}\label{prop:E(I)}
Let $S=\langle \mathbf{g}_1,\ldots,\mathbf{g}_n \rangle \subseteq \mathbb{N}^d$ be an affine semigroup and $I$ an ideal of $S$. Then, the set $f_S^{-1}(I)$ is an ideal of $\mathbb{N}^n$ and the following hold:
\begin{enumerate}
\item $S\setminus I$ is finite if and only if $\mathbb{N}^n \setminus f_S^{-1}(I)$ is finite.
\item $S\setminus I=\{f_S(\mathbf{x})\mid \mathbf{x} \in \mathbb{N}^n \setminus f_S^{-1}(I)\}$. 
\end{enumerate}
\end{proposition}
\begin{proof}
It is not difficult to see that $f_S^{-1}(I)$ is an ideal of $\mathbb{N}^n$. Moreover, since $f_S$ is a surjective function, we have $\mathbb{N}^n \setminus f_S^{-1}(I)=f_S^{-1}(S)\setminus f_S^{-1}(I)= f_S^{-1}(S\setminus I)$. Therefore, claims (1) and (2) can be proved using the same arguments of Lemma~\ref{proje}. 
\end{proof}

Denote $\operatorname{M}(I)=\operatorname{Minimals}_\leq (f_S^{-1}(I))$. Then, $f_S^{-1}(I)=\operatorname{M}(I)+\mathbb{N}^n$ and $\operatorname{M}(I)$ is a finite set. In particular, $\mathbb{N}^n \setminus f_S^{-1}(I)=\{\mathbf{x}\in \mathbb{N}^n \mid \mathbf{y}\nleq \mathbf{x}, \text{ for all }\mathbf{y}\in \operatorname{M}(I)\}$. So, in order to compute the set $S\setminus I$, it suffices to find the set $\operatorname{M}(I)$ and use equality (2) of Proposition~\ref{prop:E(I)}. In \cite[Algorithm 16]{rosales2001irreducible}, the authors show a procedure to compute the set $\operatorname{M}(I)=\operatorname{Minimals}_\leq (f_S^{-1}(I))$. We provide a different strategy to compute it.


\noindent Let $S=\langle \mathbf{g}_1,\ldots,\mathbf{g}_n \rangle \subseteq \mathbb{N}^d$ be an affine semigroup and $I=X+S$ an ideal of $S$, with $X=\{\mathbf{u}_1,\ldots,\mathbf{u}_r\}$. Observe that, $\sum_{i=1}^n a_i\mathbf{e}_i\in f_S^{-1}(I)$ if and only if there exists $k\in \{1,\ldots,r\}$ and $\lambda_1,\ldots,\lambda_n\in \mathbb{N}$ such that $\sum_{i=1}^n a_i\mathbf{g}_i=\mathbf{u}_k+\sum_{i=1}^n \lambda_i\mathbf{g}_i$. Hence, consider the matrix $A^{(S)}=[ \mathbf{g}_1,\ldots,\mathbf{g}_n, -\mathbf{g}_1,\ldots,-\mathbf{g}_n]$. For $j\in \{1,\ldots,r\}$, let $M_{(S,j)}$ be the set of non-negative integer solutions of the non homogeneous diophantine linear system $A^{(S)}\mathbf{x}=\mathbf{u}_j$. Then, we have:
$$f_S^{-1}(I)=\bigcup_{j=1}^r\left \lbrace(x_1,\ldots,x_n)\in \mathbb{N}^n \mid (x_1,\ldots,x_n,x_{n+1},\ldots,x_{2n})\in M_{(S,j)}\right\rbrace.$$
\noindent Let $M'_{(S)}$ be the set of non-negative integer solutions of the homogeneous diophantine linear system $A^{(S)}\mathbf{x}=\mathbf{0}$ and $V_{(S,j)}=\operatorname{Minimals}_{\leq}(M_{(S,j)}$) ($\leq$ is the natural partial order in $\mathbb{N}^{2n}$). In particular, we have $M_{(S,j)}=\bigcup_{\mathbf{y}\in V_{(S,j)}}(\mathbf{y}+M'_{(S)})$ (see \cite[Theorem 5.2]{pison2004nsolutions}). Let us denote $V_{(I)}=\bigcup_{j=1}^r V_{(S,j)}$. As a consequence, by the previous expression of $f_S^{-1}(I)$, we obtain:
$$\operatorname{M}(I)=\operatorname{Minimals}_{\leq} \left(\left \lbrace(x_1,\ldots,x_n)\in \mathbb{N}^n \mid (x_1,\ldots,x_n,x_{n+1},\ldots,x_{2n})\in V_{(I)}\right\rbrace \right).$$


\noindent For all $j\in \{1,\ldots,r\}$, the set $V_{(S,j)}$ can be computed in \texttt{GAP}, using the packages \texttt{numericalsgps} and \texttt{NormalizInterface}, as in Example~\ref{exa:C2}. Once the set $\operatorname{Minimals}_{\leq}(f_S^{-1}(I))$ is obtained, the following algorithm allows to check if $S\setminus I$ is finite and, if so, to compute its elements.

\begin{Algorithm}\label{alg:3}
Let $S=\langle \mathbf{g}_{1},\mathbf{g}_{2},\ldots,\mathbf{g}_{n}\rangle \subseteq \mathbb{N}^d$ be an affine semigroup and $I=X+S$ an ideal of $S$, $X=\{\mathbf{u}_1,\ldots,\mathbf{u}_r\}$. To test if $S\setminus I$ is finite and, if so, compute it, we can consider the following steps.
\begin{enumerate}
\item Consider the matrix $A^{(S)}=[\mathbf{g}_{1},\mathbf{g}_{2},\ldots,\mathbf{g}_{n}, - \mathbf{g}_{1},-\mathbf{g}_{2},\ldots,-\mathbf{g}_{n}]$, where each element is identified as a column vector.
\item For all $j\in \{1,\ldots,r\}$, compute the (finite) set $V_{(S,j)}\subseteq \mathbb{N}^{2n}$ of minimal (with respect to the natural partial order in $\mathbb{N}^{2n}$) non-negative integer solutions of the non homogeneous diophantine linear system $A^{(S)}\mathbf{x}=\mathbf{u}_j$.
\item Set $V_{(I)}=\bigcup_{j=1}^r V_{(S,j)}$.
\item Set $\operatorname{M}(I)=\operatorname{Minimals}_{\leq} \left(\left \lbrace(x_1,\ldots,x_n)\in \mathbb{N}^n \mid (x_1,\ldots,x_n,x_{n+1},\ldots,x_{2n})\in V_{(I)}\right\rbrace \right)$.
\item For all $i\in \{1,\ldots,n\}$ check if there exists $k_i\in\mathbb{N}$ such that $ k_i\mathbf{e}_i\in \operatorname{M}(I)$. If for some $i\in \{1,\ldots,n\}$ this condition does not hold, then $S\setminus I$ is not finite.
\item If the previous condition holds, then compute $Q=\{\mathbf{x}\in \mathbb{N}^n \mid \mathbf{y}\nleq \mathbf{x}, \text{ for all }\mathbf{y}\in \operatorname{M}(I)\}$.
\item Compute $S\setminus I=\{f_S(\mathbf{x})\mid \mathbf{x} \in Q\}$.
\end{enumerate}
\end{Algorithm}


Considering the computation time, the main difference between Algorithm~\ref{alg:2} and Algorithm~\ref{alg:3} concerns with the computation of the set $\operatorname{M}(I)$, against the time spent to test, for each element $\mathbf{x}=(x_1,\ldots,x_n)\in P$, if $\sum_{i=1}^n x_i \mathbf{g}_i\in I$. 

\subsection{An approach using commutative algebra}\label{subsec: commAlg}

Assume $S=\langle \mathbf{g}_1,\ldots,\mathbf{g}_n \rangle\subseteq \mathbb{N}^d$ and let $K$ be a field.  Consider the semigroup ring $K[S]=K[\mathbf{Y}^{\mathbf{s}}\mid \mathbf{s} \in S]=K[\mathbf{Y}^{\mathbf{g}_1}, \ldots, \mathbf{Y}^{\mathbf{g}_n}]$, where if $\mathbf{s}=(s_1,\ldots,s_d)\in \mathbb{N}^d$, then $\mathbf{Y}^{\mathbf{s}}=Y_1^{s_1}Y_2^{s_2}\cdots Y_d^{s_d}$. 
If $I=X+S$, $X=\{\mathbf{u}_1,\ldots,\mathbf{u}_r\}\subset S$, define $I_{K[S]}=(\mathbf{Y}^{\mathbf{u}_1},\ldots,\mathbf{Y}^{\mathbf{u}_r})$, that is a monomial ideal of $K[S]$.

\medskip
Let $S\subseteq \mathbb{N}^d$ be an affine semigroup and $I$ be an ideal of $S$. Then $\mathbf{h}\in S\setminus I$ if and only if $\mathbf{Y}^{\mathbf{h}}\notin I_{K[S]}$, that is, $\mathbf{Y}^{\mathbf{h}}$ is a monomial not belonging to $I_{K[S]}$. In particular $S\setminus I$ is a finite set if and only if the set $\{\mathbf{Y}^{\mathbf{h}}\in K[S]\mid \mathbf{Y}^{\mathbf{h}}\notin I_{K[S]}\}$ is finite. 

\medskip
Let $\mathbf{u}_i=\sum_{j=1}^n a_{ij}\mathbf{g}_j$ be a factorization of $\mathbf{u}_i$ and denote $\mathbf{a}_i=(a_{i1},\ldots,a_{in})\in \mathbb{N}^n$, for $i\in \{1,\ldots,r\}$. Consider the polynomial ring $K[Z_1,\ldots,Z_n]$ and, as above, if $\mathbf{t}=(t_1,\ldots,t_n)\in \mathbb{N}^n$, denote $\mathbf{Z}^{\mathbf{t}}=Z_1^{t_1} Z_2^{t_2} \cdots Z_n^{t_n}$. We consider the following surjective ring homomorphism:

$$ \psi: K[Z_1,\ldots,Z_n] \longrightarrow K[S]\ \ \ \ \mbox{defined by}\ \ \ \ \ Z_i \longmapsto \mathbf{Y}^{\mathbf{g}_i}.$$

Denote $J_S=\ker(\psi)$ (called the \emph{defining ideal} of $S$), and observe that $\psi(\mathbf{Z}^{\mathbf{a}_i})=\mathbf{Y}^{\mathbf{u}_i}$ for each $i\in \{1,\ldots,r\}$. The map $\psi$ induces the following ring isomorphism:

$$ \tilde{\psi}: \frac{K[Z_1,\ldots,Z_n]}{J_S} \longrightarrow K[S]\ \ \ \ \mbox{defined by}\ \ \ \ \ \tilde{\psi}(f+J_S) \longmapsto \psi(f).$$

Observe that the set $\tilde{\psi}^{-1}(I_{K[S]})$ is an ideal of $\frac{K[Z_1,\ldots,Z_n]}{J_S}$. Moreover, we have the following set equality:

$$\tilde{\psi}^{-1}(I_{K[S]})=\frac{J_S+(\mathbf{Z}^{\mathbf{a}_1}, \ldots, \mathbf{Z}^{\mathbf{a}_r})}{J_S}.$$

In fact, if $\tilde{f}=f+J_S \in \tilde{\psi}^{-1}(I_{K[S]})$, then $\tilde{\psi}(\tilde{f})\in I_{K[S]}$. In particular there exist $h_1,\ldots,h_r\in K[S]$ such that $\tilde{\psi}(\tilde{f})=\sum_{k=1}^r h_k\mathbf{Y}^{\mathbf{u}_k}=\tilde{\psi}(\sum_{k=1}^n \overline{h}_k \mathbf{Z}^{\mathbf{a}_k}+J_S)$ with $\overline{h}_k+J_S\in \tilde{\psi}^{-1}(h_k)$ for all $k\in \{1,\ldots,r\}$. Denote $\overline{h}=\sum_{k=1}^n \overline{h}_k \mathbf{Z}^{\mathbf{a}_k}$, in particular $\overline{h}\in (\mathbf{Z}^{\mathbf{a}_1}, \ldots, \mathbf{Z}^{\mathbf{a}_r})$. Since $\tilde{\psi}$ is an isomorphism we have $\tilde{f}=f+J_S=\overline{h}+J_S$, so $f-\overline{h}\in J_S$ and in particular $f\in J_S+(\mathbf{Z}^{\mathbf{a}_1}, \ldots, \mathbf{Z}^{\mathbf{a}_r})$. This means that $\tilde{f}=f+J_S\in \frac{J_S+(\mathbf{Z}^{\mathbf{a}_1}, \ldots, \mathbf{Z}^{\mathbf{a}_r})}{J_S}$. 

Conversely, if $f+J_S \in \frac{J_S+(\mathbf{Z}^{\mathbf{a}_1}, \ldots, \mathbf{Z}^{\mathbf{a}_r})}{J_S}$, then $f=g+h$ where $g\in J_S$ and $h=\sum_{k=1}^n h_k \mathbf{Z}^{\mathbf{a}_k}$, $h_k\in K[Z_1,\ldots,Z_n]$ for each $k\in \{1,\ldots,r\}$. Therefore, $\tilde{\psi}(f+J_S)=\psi(f)=\psi(g)+\psi(h)=\psi(h)=\psi (\sum_{k=1}^n h_k \mathbf{Z}^{\mathbf{a}_k})=\sum_{k=1}^n \psi(h_k) \mathbf{Y}^{\mathbf{u}_k}\in I_{K[S]}$, that is $f+J_S\in \tilde{\psi}^{-1}(I_{K[S]})$.

\medskip

As a consequence of the previous set equality, we have the following isomorphism:

$$ \frac{K[S]}{I_{K[S]}}  \cong \left( \frac{K[Z_1,\ldots,Z_n]}{J_S}\right)  \diagup\left( \frac{J_S+(\mathbf{Z}^{\mathbf{a}_1}, \ldots, \mathbf{Z}^{\mathbf{a}_r})}{J_S}\right) \cong \frac{K[Z_1,\ldots,Z_n]}{J_S+(\mathbf{Z}^{\mathbf{a}_1}, \ldots, \mathbf{Z}^{\mathbf{a}_r})}.  $$



The isomorphism above is described by the following map:


\begin{equation}\label{ismorphism}
\overline{\psi}:\frac{K[Z_1,\ldots,Z_n]}{J_S+(\mathbf{Z}^{\mathbf{a}_1}, \ldots, \mathbf{Z}^{\mathbf{a}_r})}\longrightarrow \frac{K[S]}{I_{K[S]}},\quad f+(J_S+(\mathbf{Z}^{\mathbf{a}_1}, \ldots, \mathbf{Z}^{\mathbf{a}_r}))\mapsto \psi(f)+I_{K[S]}
\end{equation}

In the following, if $J$ is an ideal of a polynomial ring $R$ and $\preceq$ is a monomial order on $R$, denote the initial ideal of $J$ with respect to $\preceq$ by $\operatorname{in}_\preceq (J)$.

\begin{lemma} \label{lem:zero-dim-idealS}
In the framework introduced above, let $P_S=J_S+(\mathbf{Z}^{\mathbf{a}_1}, \ldots, \mathbf{Z}^{\mathbf{a}_r})$ and $\preceq$ be a monomial order on $K[Z_1,\ldots,Z_n]$. Then $\{\mathbf{Y}^{\mathbf{h}}\in K[S] \mid \mathbf{Y}^{\mathbf{h}}\notin I_{K[S]}\}=\{\psi(\mathbf{Z}^{\mathbf{t}})\mid \mathbf{Z}^{\mathbf{t}}\notin \operatorname{in}_\preceq(P_S)\}$.

\end{lemma}
\begin{proof}

Let $\overline{\psi}$ be the isomorphism above. In particular, if $\mathbf{Z}^\mathbf{t}$ is a monomial in $K[Z_1,\ldots,Z_n]$, then $\overline{\psi}(\mathbf{Z}^{\mathbf{t}}+P_S)=\psi(\mathbf{Z}^{\mathbf{t}})+I_{K[S]}$. If $\mathbf{Z}^{\mathbf{t}}\notin \operatorname{in}_\preceq(P_S)$, then trivially $\mathbf{Z}^{\mathbf{t}}\notin P_S$. Since $\overline{\psi}$ is injective we obtain $\psi(\mathbf{Z}^{\mathbf{t}})+I_{K[S]}=\overline{\psi}(\mathbf{Z}^{\mathbf{t}}+P_S)\neq \mathbf{0}$, that is, $\psi(\mathbf{Z}^{\mathbf{t}})\notin I_{K[S]}$. Conversely, let $\mathbf{Y}^\mathbf{h}\in K[S]$ such that $\mathbf{Y}^\mathbf{h}\notin I_{K[S]}$. This means that $\mathbf{Y}^{\mathbf{h}}+I_{K[S]}\neq \mathbf{0}$ in $K[S]/I_{K[S]}$. Since $\mathbf{h}\in S$, there exists $\mathbf{t}=\sum_{i=1}^n t_i\mathbf{e}_i\in \mathbb{N}^n$ such that $\mathbf{h}=\sum_{i=1}^n t_i\mathbf{g}_i$. Hence, for the monomial $\mathbf{Z}^\mathbf{t}\in K[Z_1,\ldots,Z_n]$, we have $\psi(\mathbf{Z}^{\mathbf{t}})= \mathbf{Y}^{\mathbf{h}}$. Suppose $\mathbf{Z}^{\mathbf{t}}\in \operatorname{in}_\preceq(P_S)$. Then, there exists $f\in K[Z_1,\ldots,Z_n]$ such that $\mathbf{Z}^{\mathbf{t}}+f\in P_S$. So, $\mathbf{Z}^{\mathbf{t}}+f=g+\sum_{k=1}^r h_k\mathbf{Z}^{\mathbf{a}_k}$, $g\in J_S$ and $h_k\in K[Z_1,\ldots,Z_n]$. Therefore, $\psi(\mathbf{Z}^{\mathbf{t}})+\psi(f)=\psi(g)+\sum_{k=1}^r \psi(h_k)\psi(\mathbf{Z}^{\mathbf{a}_k})=\sum_{k=1}^r \psi(h_k)\mathbf{Y}^{\mathbf{u}_k}$. In particular, $\psi(\mathbf{Z}^{\mathbf{t}})+\psi(f)\in I_{K[S]}$. Since $I_{K[S]}$ is a monomial ideal and $\psi(\mathbf{Z}^{\mathbf{t}})=\mathbf{Y}^{\mathbf{h}}$ is a monomial, we obtain $\mathbf{Y}^{\mathbf{h}}\in I_{K[S]}$, a contradiction. So, we have $\mathbf{Z}^{\mathbf{t}}\notin \operatorname{in}_\preceq(P_S)$.
\end{proof}

Consider the set $\{\mathbf{Z}^{\mathbf{t}}\in K[Z_1,\ldots,Z_n] \mid \mathbf{Z}^{\mathbf{t}}\notin \operatorname{in}_\preceq(P_S)\}$ and suppose it is finite for some monomial order on $K[Z_1,\ldots,Z_n]$. This property is known as $P_S=J_S+(\mathbf{Z}^{\mathbf{a}_1}, \ldots, \mathbf{Z}^{\mathbf{a}_r})$ is a \emph{zero dimensional ideal} of $K[Z_1,\ldots,Z_n]$. A zero dimensional ideal is characterized by the following known result (Theorem 6, Chapter 5, \S 3,  \cite{CLO}):

\begin{theorem}[\cite{CLO}]\label{thm:CLO}
Let $K$ be a field and $I\subseteq K[X_1,\ldots,X_n]$ be an ideal of a polynomial ring. Fix a monomial order $\preceq$ in $K[X_1,\ldots,X_n]$. Then the following are equivalent:
\begin{enumerate}
\item[i)] For each $i\in \{1,\ldots,n\}$ there is some $m_i\geq 0$ such that $X^{m_i}\in \operatorname{in}_\preceq (I)$.

\item[ii)] Let $G$ be a Gr\"obner basis for $I$. Then for each $i\in \{1,\ldots,n\}$ there is some $m_i\geq 0$ such that $X^{m_i}=\operatorname{in}_\preceq (g)$ for some $g\in G$.

\item[iii)] The set $\{X^\alpha \mid X^\alpha \notin \operatorname{in}_\preceq (I)\}$ is finite.

\item[iv)] The $K$-vector space $K[X_1,\ldots,X_n] /I$ is finite-dimensional.

\end{enumerate}  
\end{theorem} 

Theorem~\ref{thm:CLO} suggests another way to prove the characterization of cofinitness of an ideal $I$ in a monoid $S$, given in Theorem~\ref{thm:cofinite-ideal}, and also a different procedure to compute $S\setminus I$.

\medskip
\begin{proof}[Alternative proof of Theorem~\ref{thm:cofinite-ideal}] In the following, denote $P_S=J_S+(\mathbf{Z}^{\mathbf{a}_1}, \ldots, \mathbf{Z}^{\mathbf{a}_r})$, with reference to the framework introduced in this section.

 \emph{Necessity.} Suppose that $S\setminus I$ is finite. Then the set $\{\mathbf{Y}^{\mathbf{h}}\in K[S]\mid \mathbf{Y}^{\mathbf{h}}\notin I_{K[S]}\}$ is finite. So, by Lemma~\ref{lem:zero-dim-idealS}, considering the ideal $P_S$ and $\preceq$ a monomial order on $K[Z_1,\ldots,Z_n]$, we have that $\{\psi(\mathbf{Z}^{\mathbf{t}})\mid \mathbf{Z}^{\mathbf{t}}\notin \operatorname{in}_\preceq(P_S)\}$ is a finite set. We show that the set $\{\mathbf{Z}^{\mathbf{t}}\in K[Z_1,\ldots,Z_n] \mid \mathbf{Z}^{\mathbf{t}}\notin \operatorname{in}_\preceq(P_S)\}$ is finite. If we suppose it is not finite, since $\{\psi(\mathbf{Z}^{\mathbf{t}})\mid \mathbf{Z}^{\mathbf{t}}\notin \operatorname{in}_\preceq(P_S)\}$ is finite, there exist $\mathbf{Z}^{\mathbf{t}_1}, \mathbf{Z}^{\mathbf{t}_2}\notin \operatorname{in}_\preceq(P_S)\}$, $\mathbf{Z}^{\mathbf{t}_1}\neq \mathbf{Z}^{\mathbf{t}_2}$, such that $\psi(\mathbf{Z}^{\mathbf{t}_1})=\psi(\mathbf{Z}^{\mathbf{t}_2})$. Hence, $\mathbf{Z}^{\mathbf{t}_1}-\mathbf{Z}^{\mathbf{t}_2}\in J_S\subseteq P_S$ and this implies $\mathbf{Z}^{\mathbf{t}_1}\in \operatorname{in}_\preceq(P_S)$ or $\mathbf{Z}^{\mathbf{t}_2}\in \operatorname{in}_\preceq (P_S)$, a contradiction. Therefore, the set $\{\mathbf{Z}^{\mathbf{t}}\in K[Z_1,\ldots,Z_n] \mid \mathbf{Z}^{\mathbf{t}}\notin \operatorname{in}_\preceq(P_S)\}$ is finite and by Theorem~\ref{thm:CLO} we obtain that for all $i \in \{1,\ldots,n\}$ there exists an integer $k_i\geq 0$ such that $Z_i^{k_i}\in \operatorname{in}_\preceq (P_S)$. Hence, there exists $f_i\in K[Z_1,\ldots,Z_n]$ such that $Z_i^{k_i}+f_i\in P_S$.   So, having in mind the isomorphism in (\ref{ismorphism}), we obtain $\psi(Z_i^{k_i}+f_i)=\mathbf{Y}^{k_i\mathbf{g}_i}+\psi(f_i)\in I_{K[S]}$. Since $I_{K[S]}$ is a monomial ideal and $\mathbf{Y}^{k_i\mathbf{g}_i}$ is a monomial, we obtain $\mathbf{Y}^{k_i\mathbf{g}_i}\in I_{K[S]}$, that is, $k_i\mathbf{g}_i\in I$.  As a consequence, for all $i\in \{1,\ldots,n\}$, we have $k_i\mathbf{g}_i\in I$.

\emph{Sufficiency.} For every $i\in \{1, \ldots,n\}$, suppose there exists $k_i\in \mathbb{N}\setminus \{0\}$ such that $k_i\mathbf{g}_i\in I$. Then $\mathbf{Y}^{k_i\mathbf{g}_i}\in I_{K[S]}$. In particular, by the isomorphism in (\ref{ismorphism}), we have $Z_i^{k_i}\in P_S$. By Theorem~\ref{thm:CLO}, given a monomial order $\preceq$ on $K[Z_1,\ldots,Z_n]$, we obtain that $\{\mathbf{Z}^{\mathbf{t}}\in K[Z_1,\ldots,Z_n] \mid \mathbf{Z}^{\mathbf{t}}\notin \operatorname{in}_\preceq(P_S)\}$ is a finite a set. As a consequence, the set $\{\psi(\mathbf{Z}^{\mathbf{t}})\mid \mathbf{Z}^{\mathbf{t}}\notin \operatorname{in}_\preceq(P_S)\}$ is finite and, by Lemma~\ref{lem:zero-dim-idealS}, the set  $\{\mathbf{Y}^{\mathbf{h}}\in K[S]\mid \mathbf{Y}^{\mathbf{h}}\notin I_{K[S]}\}$ is finite. This means that $S\setminus I$ is finite. 
\end{proof}

Following the arguments developed in this section, we can reformulate Theorem~\ref{thm:cofinite-ideal} as follows.

\begin{corollary}
Let $S\subseteq \mathbb{N}^d$ be an affine semigroup, $I$ an ideal of $S$. As before, let $P_S=J_S+(\mathbf{Z}^{\mathbf{a}_1}, \ldots, \mathbf{Z}^{\mathbf{a}_r})\subseteq K[Z_1,\ldots,Z_n]$. Then $S\setminus I$ is finite if and only if $P_S$ is a zero dimensional ideal of $K[Z_1,\ldots,Z_n]$.
\end{corollary}

\noindent For a different procedure to compute $S\setminus I$, recall that in the case $I\subseteq K[X_1,\ldots,X_n]$ is an ideal of a polynomial ring and $\preceq$ a monomial order in $K[X_1,\ldots,X_n]$, then the set $\{X^\alpha \mid X^\alpha \notin \operatorname{in}_\preceq (I)\}$ is a basis of $K[X_1,\ldots,X_n] /I$ as $K$-vector space. In particular one can use a computer algebra software, for instance \texttt{Macaulay2} \cite{macaulay2} or \texttt{Singular} \cite{Singular}, to test if one of the equivalent conditions of Theorem~\ref{thm:CLO} holds and, in such a case, to compute the set $\{X^\alpha \mid X^\alpha \notin \operatorname{in}_\preceq (I)\}$. 

\noindent So, if $\mathcal{B}$ is a basis of the $K$-vector space $K[Z_1,\ldots,Z_n]/P_S$, $P_S=J_S+(\mathbf{Z}^{\mathbf{a}_1}, \ldots, \mathbf{Z}^{\mathbf{a}_r})$, then $S\setminus I=\{\mathbf{h}\in \mathbb{N}^d \mid \mathbf{Y}^{\mathbf{h}}=\psi (\mathbf{Z}^{\mathbf{t}}), \mathbf{Z}^{\mathbf{t}}\in \mathcal{B}\}$.

\begin{Algorithm}\label{alg:4}\rm
Given $S=\langle \mathbf{g}_1,\ldots,\mathbf{g}_n \rangle\subseteq \mathbb{N}^d$ and an ideal $I$ of $S$, with $I=X+S$ and $X=\{\mathbf{u}_1,\ldots,\mathbf{u}_r\}\subset S$, in order to compute $S\setminus I$ we can consider the following steps:

\begin{enumerate}
\item For each $i \in \{1,\ldots, r\}$ compute a factorization $\mathbf{a}_i$ of $\mathbf{u}_i$ in $S$.
\item Set the polynomial rings $R_1=K[Z_1,\ldots,Z_n]$, $R_2=K[Y_1,\ldots, Y_d]$, with $K$ a field, the map $\psi:R_1 \rightarrow R_2$ defined by $Z_i \mapsto \mathbf{Y}^{\mathbf{g}_i}$ and compute the ideal $J_S=\ker(\psi)$.
\item Set the ideal $P_S=J_S+(\mathbf{Z}^{\mathbf{a}_1},\ldots,\mathbf{Z}^{\mathbf{a}_r})$ and compute a Gr\"obner basis $G$ of $P_S$ with respect to a monomial order $\preceq$.
\item If $G$ does not satisfies condition (ii) of Theorem~\ref{thm:CLO}, then $S\setminus I$ is not finite and we can stop. Otherwise compute a basis $\mathcal{B}$ of the $K$-vector space $K[Z_1,\ldots,Z_n]/P_S$.
\item Compute $S\setminus I=\{\mathbf{h}\in \mathbb{N}^d \mid \mathbf{Y}^{\mathbf{h}}=\psi (\mathbf{Z}^{\mathbf{t}}), \mathbf{Z}^{\mathbf{t}}\in \mathcal{B}\}$.

\end{enumerate}

\end{Algorithm}

We point out that a similarity between Algorithm~\ref{alg:3} and Algorithm~\ref{alg:4} is actually hidden. That is, the exponent vectors of elements in the basis $\mathcal{B}$ of the $K$-vector space $K[Z_1,\ldots,Z_n]/P_S$ correspond to the vectors in $\mathbb{N}^n \setminus f_S^{-1}(I)$. In particular, these elements are obtained from a \emph{presentation} of $S$ (see \cite{rosales2001irreducible} for more details) in the first algorithm, and from the \emph{defining ideal} $I_S$ of $K[S]$ in the second algorithm. Considering the computation time, the relevant difference concerns with the time spent to compute the set $\operatorname{M}(I)$ against the computation of a Gr\"obner basis of the ideal $P_S$.


\subsection{A remark on Ap\'ery sets}\label{sec3}

Let $S\subseteq \mathbb{N}^d$ be an affine semigroup and $X\subseteq S$. The Ap\'ery set of $S$ with respect to $X$ is defined as $\operatorname{Ap}(S,X)=\{\mathbf{s}\in S\mid \mathbf{s}-\mathbf{x} \notin S\text{ for all }\mathbf{x}\in X\}$. This set is an important tool in the context of affine semigroups. For instance, in the case $S$ is \emph{simplicial}, it can be used to verify the Cohen-Macaulay and Gorenstein conditions for the associated semigroup ring (see \cite{rosales1998cohen}) and to compute the conductor (see \cite{jafari2022type}). Observe that $\operatorname{Ap}(S,X)=S\setminus (X+S)$, so it is the complement in $S$ of the ideal $X+S$. In particular we can state the following:

\begin{corollary}
Let $S\subseteq \mathbb{N}^d$ be an affine semigroup minimally generated by the set $\{\mathbf{g}_1,\ldots,\mathbf{g}_n\}$. If $X\subset S$, then $\operatorname{Ap}(S,X)$ is finite if and only if for all $i\in \{1,\ldots,n\}$ there exists $k_i\in \mathbb{N}\setminus\{0\}$ such that $k_i\mathbf{g}_i\in X+S$.
\end{corollary}

As a consequence, if $S=\langle \mathbf{g}_{1},\mathbf{g}_{2},\ldots,\mathbf{g}_{r},\mathbf{g}_{r+1},\mathbf{g}_{r+2},\ldots,\mathbf{g}_{r+m} \rangle$, where $E=\{\mathbf{g}_{1},\mathbf{g}_{2},\ldots,\mathbf{g}_{r}\}$ is a set of extreme rays of $S$, then $\mathbf{g}_j=\sum_{i=1}^r q_i \mathbf{g}_i$, $q_i\in \mathbb{Q}_+$, for all $j\in \{1,\ldots,m\}$. It easily follows that there exists $k_j\in\mathbb{N}$ such that $k_j\mathbf{g}_j \in \langle \mathbf{g}_{1},\mathbf{g}_{2},\ldots,\mathbf{g}_{r} \rangle \subset S$. Therefore, in the case $E$ is a set of extreme rays of $S$, the set $\operatorname{Ap}(S,E)$ is finite. Moreover, referring to the previous arguments in Subsection~\ref{subsec: commAlg}, we can consider the ideal $P_S=J_S+(Z_1,\ldots, Z_r)$ and if $\mathcal{B}$ is a basis of the $K$-vector space $K[Z_1,\ldots,Z_n]/P_S$, then $\operatorname{Ap}(S,E)=\{\mathbf{h}\in \mathbb{N}^d \mid \mathbf{Y}^{\mathbf{h}}=\psi (\mathbf{Z}^{\mathbf{t}}), \mathbf{Z}^{\mathbf{t}}\in \mathcal{B}\}$, obtaining the same result contained in \cite[Theorem 3.3]{ojeda2017short}.

\section*{Acknowledgements}

Motivation for this paper was inspired by some discussions had with P. A. Garc{\'\i}a-S\'anchez during a period the author spent at the University of Granada. The author would like to express his gratitude to him, for the hospitality and for his very helpful comments and suggestions, that have allowed to improve this work. The author also acknowledges support from the Institute of Mathematics of the University of Granada (IMAG), through the program of Visits of Young Talented Researchers and from Istituto Nazionale di Alta Matematica (INDAM), through the program Concorso a n. 30 mensilità di borse di studio per l’estero per l’a.a. 2022-2023.

\end{document}